\documentclass[11pt,a4paper,oneside]{amsart}
\usepackage[a4paper,margin=2.5cm]{geometry}
\usepackage{mathrsfs}
\usepackage{mathtools}
\usepackage{color}
 
\theoremstyle{plain}                 
\newtheorem{theorem}{Theorem}[section]     
\newtheorem{proposition}[theorem]{Proposition} 
\newtheorem{corollary}[theorem]{Corollary}     
\newtheorem{lemma}[theorem]{Lemma}        
\theoremstyle{definition}           
\newtheorem{definition}{Definition}    
\newtheorem{example}{Example} 
\theoremstyle{remark}       
\newtheorem{remark}{Remark}

\newcommand{\cor}[1]{\bigg{\langle} \,  #1 \, \bigg{\rangle}}
\newcommand{\corc}[1]{\bigg{\langle} \, #1 \,  \bigg{\rangle} ^{\circ}}
\newcommand{\cord}[1]{\bigg{\langle} \, #1 \, \bigg{\rangle} ^{\bullet}}
\newcommand{\corgw}[1]{\bigg{\langle} \,  #1 \, \bigg{\rangle}^{\mathbb{P}^1}}
\newcommand{\corcgw}[1]{\bigg{\langle} \, #1 \,  \bigg{\rangle}^{\mathbb{P}^1,  \circ}}
\newcommand{\cordgw}[1]{\bigg{\langle} \, #1 \, \bigg{\rangle}^{\mathbb{P}^1, \bullet}}

\DeclareMathOperator\Aut{Aut}
\DeclareMathOperator\val{val}
\DeclareMathOperator\RHS{RHS}

\newcommand{\sta}{\mathcal}

\newcommand{\MMMbar}{\overline{\sta M}}

\def\Z{\mathbb{Z}}
\def\N{\mathbb{N}}
\def\R{\mathbb{R}}

\def\C{\mathbb{C}}

\def\CP1{\mathbb{C}\mathrm{P}^1}

\def\SS{\mathcal{S}}
\def\FF{\mathcal{F}}

\def\EE{\mathcal{E}}

\newcommand{\mc}[1]{\mathcal{#1}}

\begin{document}

\title[Tropical Jucys covers]{Tropical Jucys covers}
\author[M.~A.~Hahn]{Marvin Anas Hahn}
\address{M.~A.~H.: Mathematisches Institut, Universit\"at T\"ubingen, Auf der Morgenstelle 10, 72076 T\"ubingen, Germany.}
\email{marvinanashahn@gmail.com}
\author[D.~Lewanski]{Danilo Lewanski}
\address{D.~L.: Max Planck Institut f\"{u}r Mathematik, Vivatsgasse 7, 53111 Bonn, Germany.}
\email{ilgrillodani@mpim-bonn.mpg.de}

\begin{abstract}
We study monotone and strictly monotone Hurwitz numbers from a bosonic Fock space perspective. This yields to an interpretation in terms of tropical geometry involving local multiplicities given by Gromov-Witten invariants. Furthermore, this enables us to prove that a main result of Cavalieri-Johnson-Markwig-Ranganathan is a actually a refinement of the Gromov-Witten/Hurwitz correspondence by Okounkov-Pandharipande.
\end{abstract}

\maketitle
\tableofcontents
\section{Introduction}
Hurwitz numbers and Gromov-Witten invariants with target the Riemann sphere enumerate certain maps between Riemann surfaces. The last two decades have shown several fruitful interactions between those two notions. One of the key revelations for these interactions is the so-called \textit{Gromov-Witten/Hurwitz (GW/H) correspondence} -- found by Okounkov and Pandharipande in \cite{OP} -- which is a substition rule that relates (completed cycles) Hurwitz numbers and stationary descendant Gromov-Witten invariants. One of the key components of the GW/H correspondence is the notion of the \textit{Fock space}, which gives an operator theoretic interpretation of Gromov-Witten invariants and Hurwitz numbers (for more, see \cite{Johnson,OP2}).\par
One successful approach to Hurwitz numbers has been achieved by means of tropical geometry, in which Hurwitz numbers are expressed as enumeration of maps between metric graphs with discrete data (tropical covers) \cite{CJM,CJMa,BBM}. This interpretation has given rise to many interesting insights, such as a study of polynomial behaviour of double Hurwitz numbers \cite{CJMa} or the introduction of tropical mirror symmetry for elliptic curves involving quasi-modularity statements of certain tropical covers \cite{boehm2017tropical}. Furthermore, it was observed in \cite{CJMR} that tropical geometry gives rise to a graphical interface for the Gromov-Witten theory of curves. In particular, it was proved for the case of stationary descendant Gromov-Witten invariants with target $\mathbb{P}^1$ that the Fock space notion in \cite{OP} is strongly related to the tropical expression derived in \cite{CJMR}. The tropical covers are weighted by local multiplicities, which are Gromov-Witten invariants themselves.\par
In recent years many variants of Hurwitz numbers have started to appear in the literature and have grown in importance, also in relation with Chekhov-Eynard-Orantin topological recursion theory. Among those we focus on two, namely \textit{monotone} and \textit{strictly monotone Hurwitz numbers}. Monotone Hurwitz numbers appear in the context of random matrix theory as coefficients of large $N$ expansion of the Harish-Chandra-Itzykson-Zuber model, whereas strictly monotone Hurwitz numbers enumerate certain types of Grothendieck dessins d'enfant. These numbers have been intensively studied in terms of operators acting on the Fock space. Furthermore, special cases of these enumerations may be expressed by means of tropical geometry as well \cite{DK,hahn2017monodromy}. These expressions however differ fundamentally from the tropical covers involved in \cite{CJMR}. It is natural to ask whether there exists a tropical interpretation involving local Gromov-Witten multiplicities as in the classical case \cite{CJMR}.\par 
The aim of this paper is three-fold. Starting from previous work on monotone and strictly monotone Hurwitz numbers in terms of the fermionic Fock space \cite{ALS,HKL}, we derive an expression of these variants of Hurwitz numbers in terms of the bosonic Fock space.\par 
We then use the connection between bosonic operators and tropical geometry studied in \cite{BG,CJMR} to answer the previously posed question positively. In fact we derive a new tropical expression for monotone and strictly monotone Hurwitz numbers that shares several features with the results of \cite{CJMR}, as the tropical covers are once again weighted by local multiplicities given by Gromov-Witten invariants. This new interpretation has the advantage that the curves involved carry less non-geometric information than the ones in \cite{DK,hahn2017monodromy} and are closer to the tropical curves involved in \cite{CJMa,boehm2017tropical}.\par 
Finally, we use our methods to compare the work of \cite{CJMR} to the work in \cite{OP} and prove that theorem 5.3.4 in \cite{CJMR} is a refinement of the original GW/H correspondence.

\subsection{Structure of the paper}
In section \ref{sec:pre}, we introduce the basic notions revolving around the relation between Hurwitz numbers, Gromov-Witten theory, the Fock space and tropical geometry needed for our work. In section \ref{sec:bos} we derive a bosonic expression for monotone and strictly monotone Hurwitz numbers from the fermionic one. Our main result lies in section \ref{sec:trop}, where we express monotone and strictly monotone Hurwitz numbers in terms of tropical covers with local multiplicities given by Gromov-Witten invariants. Lastly, in section \ref{sec:grom} we analyse and compare the original \cite{OP} and the tropical \cite{CJMR} version of Gromov-Witten/Hurwitz correpondence via the semi-infinite wedge formalism.

\subsection{Acknowledgments}
We would like to thank Hannah Markwig for useful discussions. The work of D.L. is supported by the Max-Planck-Gesellschaft. M.A.H. gratefully acknowledges partial support by DFG SFB-TRR 195 “Symbolic tools in mathematics and their applications”, project A 14 “Random matrices and Hurwitz numbers” (INST 248/238-1).

\section{Preliminaries}
\label{sec:pre}
In this section, we recall the basic notions regarding the Fock space and tropical geometry, needed for our work. A concise introduction to the needed Gromov-Witten theory can be found in \cite{CJMR} section 2.

\subsection{Semi-infinite wedge formalism}

We introduce the operators needed for the derivation of our results. For a self-contained introduction to the infinite wedge space formalism, we refer the reader to \cite{OP,Johnson}, where most relevant objects are defined. 
\par
Let us define $\varsigma(z) = 2\sinh(z/2) = e^{z/2} - e^{-z/2}$ and $\SS(z) = \varsigma(z)/z$. We are particularly interested in the expansion of
\begin{equation}
\frac{1}{\varsigma(z)} = \frac{1}{2 \sinh(z/2)} = \sum_{l=0} c_{2l-1} z^{2l-1} = \frac{1}{z} - \frac{1}{24}z + \frac{7}{5760}z^3 + O(z^5)
\end{equation}
\noindent
The coefficients $c_l$ have several well known combinatorial and geometric interpretations. The first coefficient is $c_{-1} = 1$, and \begin{equation}\label{eq:formscl}
c_{2l-1}  \;\; = \;\; - \frac{(2^{2l-1} - 1)}{2^{2l-1}}\frac{B_{2l}}{(2l)!}
\;\; = \;\;
  (-1)^{l} \int_{\MMMbar_{l,1} } \!\!\!\lambda_l\psi_1^{2l-2}
\;\; = \;\;
  \left \langle \tau_{2l-2}(\omega) \right \rangle_{l,1}^{\mathbb{P}^1}, \qquad \text{ for   } l>0,
\end{equation}
where:
\begin{enumerate}
\item[\textit{i).}] $B_k$ is the $k$-th Bernoulli number, defined by the generating series $\frac{t}{e^t - 1} = \sum_{k=0}^{\infty} B_k \frac{t^k}{k!}$.
\item[\textit{ii).}] $\MMMbar_{l,1}$ is the moduli space of stable curves of genus $l$ and one marked point $p$. The class $\lambda_l := c_l(\mathbb{E})$ is the top Chern class of the Hodge bundle $\mathbb{E} \rightarrow \MMMbar_{l,1}$, which is the rank $l$ vector bundle with fiber $H^0(C, \omega_C)$ over $[C, p]$. The class $\psi_1 = c_1(\mathbb{L})$ is the first Chern class of the cotangent line bundle $\mathbb{L} \rightarrow \MMMbar_{l,1}$ that has fiber $T^{*}_{C,p}$ over the moduli point $[C,p]$.
\item[\textit{iii).}] Recall that for two partitions $\nu, \mu$ of same size $d= |\mu| = |\nu|$, the relative Gromov-Witten invariant with target $\mathbb{P}^1$ is defined as
$$
\corgw{\nu,  \tau_{k_1}(\omega) \tau_{k_2}(\omega) \cdots \tau_{k_n}(\omega) ,\mu}_{g,n} := \int_{[\overline{\mathcal{M}}_{g,n}(\mathbb{P}^1, \nu, \mu, d)]^{vir}} \prod_{i=1}^n ev_i^*(\omega) \psi_i^{k_i},
$$
where $\overline{\mathcal{M}}_{g,n}(\mathbb{P}^1, \nu, \mu, d)$ is the moduli space of stable curves over $\mathbb{P}^1$ relative to the partitions $\nu$ and $\mu$, $\psi_i$ is again $c_1$ of the cotangent line bundle over the $i$-th marked point, the map $ev_i: \overline{\mathcal{M}}_{g,n}(\mathbb{P}^1, \nu, \mu, d) \rightarrow \mathbb{P}^1$ is the $i$-th evaluation morphism that sends the moduli point $[C, p_1, \dots, p_n, f]$ to $f(p_i)$, with $f:C \rightarrow \mathbb{P}^1$ the stable relative map, and, finally,  $\omega$ denotes the class of a point in $\mathbb{P}^1$. For a complete introduction on the topic we recommend \cite{V}.
When the partitions are omitted and the degree is not specified, degree zero and empty partitions are meant. The subscripts "$\circ$" or "$\bullet$" refer to the \emph{connected} or the \emph{not necessarily connected} (also called disconnected for simplicity) Gromov-Witten invariant, and correspond in the definition to connected or not necessarily connected (disconnected) stable maps in the moduli space, respectively.
\end{enumerate}

Let $V = \bigoplus_{i \in \Z + 1/2} \C \underline{i}$ be an infinite-dimensional complex vector space with a basis labeled by half-integers, written as \( \underline{i}\). The semi-infinite wedge space or Bosonic Fock space \( \mc{V} := \bigwedge^{\frac{\infty}{2}} V \) is the space spanned by vectors
\begin{equation}
\underline{k_1} \wedge \underline{k_2} \wedge \underline{k_3} \wedge \cdots
\end{equation}
such that for large \( i\), \( k_i + i -\frac{1}{2} \) equals a constant, called the charge, imposing that \( \wedge \) is antisymmetric. The charge-zero sector 
\begin{equation}
\mathcal{V}_0 = \bigoplus_{n \in \mathbb{N} } \bigoplus_{\lambda\,  \vdash\, n} \C v_{\lambda}
\end{equation}
is then the span of all of the semi-infinite wedge products \( v_{\lambda} = \underline{\lambda_1 - \frac{1}{2}} \wedge \underline{\lambda_2 - \frac{3}{2}} \wedge \cdots \) for integer partitions~$\lambda$. The space $\mathcal{V}_0$ has a natural inner product $(\cdot,\cdot )$ defined by declaring its basis elements to be orthonormal. 
The element corresponding to the empty partition $v_{\emptyset}$ is called the vacuum vector and denoted by $|0\rangle$. Similarly, we call the covacuum vector its dual in \( \mc{V}_0^* \), and denote it by $\langle0|$. If $\mathcal{P}$ is an operator acting on $\mathcal{V}_0$, we denote with $ \langle \mathcal{P}\rangle^{\bullet}$ the evaluation $ \langle 0 |  \mathcal{P} |0\rangle$.
\noindent
For $k$ half-integer, define the operator $\psi_k$ by $\psi_k : (\underline{i_1} \wedge \underline{i_2} \wedge \cdots) \ \mapsto \ (\underline{k} \wedge \underline{i_1} \wedge \underline{i_2} \wedge \cdots)$, and let $\psi_k^{\dagger}$ be its adjoint operator with respect to~$(\cdot,\cdot)$. The normally ordered products of $\psi$-operators
\begin{equation}
E_{i,j} \coloneqq \begin{cases}\psi_i \psi_j^{\dagger}, & \text{ if } j > 0 \\
-\psi_j^{\dagger} \psi_i & \text{ if } j < 0 \end{cases} 
\end{equation}
are well-defined operators on $\mathcal{V}_0$. 
For $n$ any integer, and $z$ a formal variable, define the operators
\begin{equation}
\mathcal{E}_n(z) = \!\!\!\! \sum_{k \in \Z + \frac12} \!\!\! e^{z(k - \frac{n}{2})} E_{k-n,k} + \frac{\delta_{n,0}}{\varsigma(z)}, \qquad \tilde{\mathcal{E}}_0(z) = \!\!\!\! \sum_{k \in \Z + \frac12} \!\!\! e^{zk} E_{k,k}, \qquad  \alpha_n = \mathcal{E}_n(0) = \!\!\!\! \sum_{k \in \Z + \frac12} \!\!\! E_{k-n,k}.
\end{equation}
Their commutation formulae are known to be
\begin{equation}
  \left[\mathcal{E}_a(z),\mathcal{E}_b(w)\right] =
\varsigma (aw-bz)
\,
\mathcal{E}_{a+b}(z+w)
\; \text{ for } z \neq 0 \text{ or } w \neq 0, 
\end{equation}
and otherwise
\begin{equation}\label{eq:commalphas}
[\alpha_k, \alpha_l ] = k \delta_{k+l,0}.
\end{equation}
The commutation relations involving $\tilde{\mathcal{E}}_0$ are clearly the same as the ones involving $\mathcal{E}_0$.
The coefficients of both operators will play a crucial role in the rest of the paper:
\begin{equation}
\FF_l  := [z^l] \tilde{\mathcal{E}}_0(z) = \sum_{k \in \Z + 1/2} \frac{k^l}{l!} E_{k,k}, \qquad \qquad 
\FF_l^{sh}  := [z^l] \mathcal{E}_0(z) = \sum_{k \in \Z + 1/2} \frac{k^l}{l!} E_{k,k} + c_l.
\end{equation}
The operator \( C = \mc{F}_0 \) is called the \emph{charge operator}, as its eigenvalues on basis vectors are given by the charge. In particular, it acts as zero on \( \mc{V}_0 \). The operator $E = \mathcal{F}_1$ is called the \emph{energy operator}. Observe that the commutation relation for $\alpha$ operators implies that
\begin{equation}\label{eq:pairing}
\cord{
\prod_{j=1}^{\ell(\mu)} \alpha_{\mu_j}
\prod_{i=1}^{\ell(\nu)} \alpha_{-\nu_i} 
}
=
|\mathrm{Aut}(\mu)|\prod_{i=1}^m\mu_i \cdot \delta_{\mu,\nu},
\end{equation}
where $|\Aut(\mu)| = \prod_{i=1}^{\infty} m_i(\mu)!$, and $m_i(\mu)$ the number of parts equal to $i$ in the partition $\mu$.

\subsection{Hurwitz numbers in the semi-infinite wedge formalism}
We are now ready to express the Hurwitz numbers in terms of the semi-infinite wedge formalism. The monotone and strictly monotone Hurwitz numbers have respectively the following expressions, derived in \cite{ALS}. We use these expressions as definitions. For $g \in \N$, two partitions $\mu$ and $\nu$, $m = 2g - 2 + \ell(\mu) + \ell(\nu)$, let
\begin{align}
h_{g;\mu,\nu}^{\le, \bullet}&:=\frac{[u^m]}{\prod\mu_i\prod\nu_j}\left\langle \prod_{i=1}^m \alpha_{\mu_i} \mathcal{D}^{(h)}(u) \prod_{j=1}^n\alpha_{-\nu_j}\right\rangle^{\bullet}\,,
\label{eq:hfockh}
\\
h_{g;\mu,\nu}^{<, \bullet}&:=\frac{[u^m]}{\prod\mu_i\prod\nu_j}\left\langle \prod_{i=1}^m \alpha_{\mu_i} \mathcal{D}^{(\sigma)}(u) \prod_{j=1}^n\alpha_{-\nu_j}\right\rangle^{\bullet}\,,\label{eq:hfocksigma}
\end{align}
where the operators $\mathcal{D}^{(h)}(u), \, \mathcal{D}^{(\sigma)}(u)$ depending on the formal variable $u$ have the vectors $v_{\lambda}$ as eigenvectors with eigenvalues the generating series of complete homogeneous polynomials $h$ and elementary symmetric polynomials $\sigma$, evaluated in the content $\mathbf{cr}^{\lambda}$ of the Young tableau associated to the partition $\lambda$:
\begin{equation}
\mathcal{D}^{(h)}(u).v_{\lambda} = \sum_{v=0} h_v(\mathbf{cr}^{\lambda}) u^v v_{\lambda}\,,
\qquad
\mathcal{D}^{(\sigma)}(u).v_{\lambda} = \sum_{v=0} \sigma_v(\mathbf{cr}^{\lambda}) u^v v_{\lambda}\,.
\end{equation}
Remember that the content of a box \( (i,j) \) in the Young tableau of a partition is given by \( \mathbf{cr}_{(i,j)} = j-i \), and the content \( \mathbf{cr}^\lambda \) of a partition \( \lambda \) is the multiset of all contents of boxes in its Young diagram (\( \mathbf{cr} \) stands for \textbf{c}olumn-\textbf{r}ow). For example, the partition \( (3,2) \) has boxes \( (1,1)\), \( (1,2)\), \( (1,3)\), \((2,1)\), and \( (2,2) \), so \( \mathbf{cr}^{(3,2)} = \{ 0,1,2,-1,0\} \). The connected monotone and strictly monotone Hurwitz numbers are defined in the same way by taking the connected vacuum expectation on the right-hand side. The connected expectation is defined from the disconnected one by means of the inclusion-exclusion formula.
\par

\subsection{Tropical curves}
We introduce now the concepts of tropical curve and tropical cover, and we describe the relation between Fock space vacuum expectations and their tropical counterparts. This relation is sometimes indicated in tropical geometry with the motto
$$
\textit{``Bosonification is Tropicalisation."}
$$

A detailed introduction to tropical covers can be found in \cite{ABBR}.

\begin{definition}
\label{def:abstrop}
An \textit{abstract tropical curve} is a connected
metric graph $\Gamma$ with unbounded edges called ends, together with a function associating a genus $g(v)$ to each vertex $v$. Let $V(\Gamma)$ be the set of its vertices. Let $E(\Gamma)$ and $E'(\Gamma)$ be the set of its internal (or bounded) edges and its set of all edges, respectively. The set of ends is therefore $E'(\Gamma) \setminus E(\Gamma)$, and all ends are considered to have infinite length. The genus of an abstract tropical curve $\Gamma$ is
 $ g(\Gamma)\coloneqq h^1(\Gamma) + \sum_{v \in V(\Gamma)} g(v)$,
where $h^1(\Gamma)$ is the first Betti number of the underlying graph.
An \textit{isomorphism} of a tropical curve is an automorphism of the underlying graph that respects edges' lengths and vertices' genera.
The \textit{combinatorial type} of a tropical curve is obtained by disregarding its metric structure.
\end{definition}

\begin{definition}
\label{def:tropmorph}
A tropical cover is a surjective harmonic map $\pi:\Gamma_1\to\Gamma_2$ between abtract tropical curves as in \cite{ABBR}, i.e.:
\begin{itemize}
\item[\textit{i).}] Let $V(\Gamma_i)$ denote the vertex set of $\Gamma_i$, then we require $\pi(V(\Gamma_1))\subset V(\Gamma_2)$;
\item[\textit{ii).}] Let $E'(\Gamma_i)$ denote the edge set of $\Gamma_i$, then we require $\pi^{-1}(E'(\Gamma_2))\subset E'(\Gamma_1)$;
\item[\textit{iii).}] For each edge $e\in E'(\Gamma_i)$, denote by $l(e)$ its length. We interpret $e\in E'(\Gamma_1),\pi(e)\in E'(\Gamma_2)$ as intervals $[0,l(e)]$ and $[0,l(\pi(e))]$, then we require $\pi$ restricted to $e$ to be a linear map of slope $\omega(e)\in\mathbb{Z}_{\ge0}$, that is $\pi:[0,l(e)]\to[0,l(\pi(e))]$ is given by $\pi(t)=\omega(e)\cdot t$. We call $\omega(e)$ the \textit{weight} of $e$. If $\pi(e)$ is a vertex, we have $\omega(e)=0$.
\item[\textit{iv).}] For a vertex $v\in\Gamma_1$, let $v'=\pi(v)$. We choose an edge $e'$ adjacent to $v'$. We define the local degree at $v$ as
\begin{equation}
d_v=\sum_{\substack{e\in\Gamma_1\\\pi(e)=e'}}\omega_e.
\end{equation}
We require $d_v$ to be independent of the choice of edge $e'$ adjacent to $v'$. We call this fact the \textit{balancing} or \textit{harmonicity condition}.
\end{itemize}

We furthermore introduce the following notions:
\begin{itemize}
\item[\textit{i).}] The \textit{degree} of a tropical cover $\pi$ is the sum over all local degrees of pre-images of any point in $\Gamma_2$. Due to the harmonicity condition, this number is independent of the point in $\Gamma_2$.
\item[\textit{ii).}] For any end $e$, we define a partion $\mu_e$ as the partition of weights of the ends of $\Gamma_1$ mapping to $e$. We call $\mu_e$ the \textit{ramification profile} above $e$.
\end{itemize}
\end{definition}

We give in the following a formulation of Wick's theorem that suits better our purposes. In fact, Wick's theorem is generally expressed in terms of Feynman graphs, which can then be regarded as tropical covers (for more details see, e.g., \cite{BG}, last section). We state the result in terms of tropical covers, directly.

\begin{proposition}[Wick's theorem]
\label{prop:wick}
Let $\mu$ and $\nu$ be two partitions of the same size. Consider any collection of non-empty finite sets of non-zero integers $\textbf{x}_1, \dots, \textbf{x}_n$ s.t. for each $j$ we have $\sum_{x_{i,j} \in \textbf{x}_{j}} x_{i,j} = 0$. Denote by $x_i^+$ ($x_i^-$) the tuple of positive (resp. negative) entries of $x_i$. Then the vacuum expectation
\begin{equation}
\cor{
\prod_{t=1}^{\ell(\mu)} \alpha_{\mu_t} 
\prod_{0 > x_{i,1} \in \textbf{x}_1} \!\!\!\!\!\!\alpha_{x_{i,1}}
\prod_{0 < x_{i,1} \in \textbf{x}_1} \!\!\!\!\!\!\alpha_{x_{i,1}}
\cdots \cdots
\prod_{0 > x_{i,n} \in \textbf{x}_n} \!\!\!\!\!\!\alpha_{x_{i,n}}
\prod_{0 < x_{i,n} \in \textbf{x}_n} \!\!\!\!\!\!\alpha_{x_{i,n}}
\prod_{j=1}^{\ell(\nu)} \alpha_{-\nu_j}
} 
\end{equation}
is equal to 
\begin{equation}
\label{equ:trop}
\sum_{\pi \in \Gamma( \mathbb{P}^1_{\text{trop}}; \mu, \nu)}\frac{1}{|\mathrm{Aut}(\pi)|} \prod_{i=1}^n|\mathrm{Aut}(x_i^+)||\mathrm{Aut}(x_i^-)|\prod_{e \in E(\Gamma)} \omega_e,
\end{equation}
where $\Gamma( \mathbb{P}^1_{\text{trop}}; \mu, \nu)$ is the set of  tropical covers 
$
\pi: \Gamma \longrightarrow \mathbb{P}^1_{trop} = \R
$
 such that

\begin{itemize}
\item[\textit{i).}] The unbounded left (resp. right) pointing ends of $\Gamma$ have weights given by the partition $\mu$ (resp. $\nu$).
\item[\textit{ii).}] $|V(\Gamma)| = n$. Let $\{v_1, \dots, v_n\}$ be the set of its vertices ordered linearly from left to right. The local structure at $v_j$ is determined by $\textbf{x}_j$. For each operator $\alpha_{x_{j,i}}$, we draw an edge germ of weight $|x_{j,i}|$, which points to the left for $x_{j,i}<0$ and to the right $x_{j,i}>0$.
In particular, the valence $ \val(v_j)$ of $v_j$ is equal to $|\textbf{x}_j|$.
\item[\textit{iii).}] $|E(\Gamma)| = \sum_j |\textbf{x}_j|/2$. Every element of $E(\Gamma)$ is formed  by connecting one edge germ pointing to the right with one edge germ with same weight but pointing to the left. Viceversa, every edge germ that does not correspond to an unbounded edge (i.e. does not correspond to an element of $\nu$ or $\mu$) must be matched to another edge germ of the same weight $\omega_e$ to form an internal edge $e \in E(\Gamma)$. This means in particular that each internal operator $\alpha_{x_{j,i}}$ must be matched with an operator $\alpha_{x_{k,l} = -x_{j,i}}$.
\item[\textit{iv).}] For the genus of each vertex, we have $g(v_i)=0$.
\end{itemize}
\end{proposition}

\section{Bosonification}
In this section we derive a bosonic expression for monotone and strictly monotone Hurwitz numbers from the fermionic one. This is done by means of the Boson-Fermion correspondence. The fermion expression itself is recovered from the fermionic expression for the power sums via transformations at the level of symmetric functions.
\label{sec:bos}
\subsection{Newton's identities: $\sigma$ and $h$ polynomials in terms of power sums $p$}
\label{sec:Newton}
Let $\textbf{X}$ be the set of variables $\{X_1, \dots, X_n\}$. We indicate with $\sigma_m, \, h_m$ and $p_m$ the symmetric elementary polynomials, the complete homogeneous polynomials and the power sums, respectively:

\begin{equation*}
\sigma_m(\textbf{X}) = \!\!\!\!\!\!\!\! \sum_{1 \leq i_1 < \dots < i_m \leq n} X_{i_1}\cdots X_{i_m}
\qquad 
\qquad
h_m(\textbf{X}) = \!\!\!\!\!\!\!\!\sum_{1 \leq i_1 \leq \dots \leq i_m \leq n} X_{i_1}\cdots X_{i_m}
\qquad 
\qquad
p_m(\textbf{X}) = \sum_{i=1}^n X_i^m
\end{equation*}
\noindent
Newton identities describe relations between the power sums $p_m$ and bases $\sigma_m$ and $h_m$:
\begin{align}
\label{Newtonsigma}
\sigma_m(\textbf{X}) & = [z^m]. \exp \left( - \sum_{i \geq 1 } \frac{p_i (\textbf{X})}{i} (-z)^i \right),
\\ 
\label{Newtonh}
h_m(\textbf{X}) & = [z^m]. \exp \left(\sum_{i \geq 1 } \frac{p_i (\textbf{X}) }{i} z^i \right).
\end{align}

\begin{example} Let us test Newton identities for $m=2$ and $n=3$. From the right-hand side of \eqref{Newtonsigma} we compute
\begin{align*}
\RHS\eqref{Newtonsigma}=&\frac{(-1)^3}{1!} \frac{p_2(X_1, X_2, X_3)}{2} + \frac{(-1)^2}{2!}p_1^2(X_1, X_2, X_3) 
= \frac{1}{2}\left[ -X_1^2 - X_2^2 - X_3^2 + (X_1 + X_2 + X_2)^2\right] \\
=& X_1X_2 + X_1X_3 + X_2X_3,
\end{align*}
which is indeed equal to $\sigma_2(X_1, X_2, X_3)$. From the right-hand side of \eqref{Newtonh} we compute
\begin{align*}
\RHS\eqref{Newtonh} = &\frac{1}{1!} \frac{p_2(X_1, X_2, X_3)}{2} + \frac{1}{2!}p_1^2(X_1, X_2, X_3) 
= \frac{1}{2}\left[ X_1^2 + X_2^2 + X_3^2 + (X_1 + X_2 + X_2)^2\right] \\
=& X_1X_2 + X_1X_3 + X_2X_3  + X_1^2 + X_2^2 + X_3^2 ,
\end{align*}
which is indeed equal to $h_2(X_1, X_2, X_3)$.
\end{example}

We can write it more explicitly in terms of (ordered) partitions $\lambda$. We have
\begin{align*}
h_m & = [z^m]. \exp \left(\sum_{i \geq 1 } \frac{p_i}{i} z^i \right) =  \sum_{\lambda \vdash m} \frac{1}{\ell(\lambda)!} \prod_i^{\ell(\lambda)} \frac{p_{\lambda_i}}{\lambda_i}\quad,
\\
\sigma_m & = [z^m]. \exp \left( - \sum_{i \geq 1 } \frac{p_i}{i} (-z)^i \right) = \sum_{\lambda \vdash m} \frac{(-1)^{m + \ell(\lambda)}}{\ell(\lambda)!} \prod_i^{\ell(\lambda)} \frac{p_{\lambda_i}}{\lambda_i}\quad.
\end{align*}

\subsection{Lascoux-Thibon operator: power sums $p$ of the content in terms of $\mathcal{F}$ operators}
\label{sec:LascouxThibon}

Let $\mathbf{cr}^{\lambda}$ be the content of the Young diagram associated to the partition $\lambda$, let $p_k$ be the $k$-th power sum. We moreover adopt the convention $\FF_r = 0$ for $r\leq 0$, as our calculations take all place in the charge zero sector.
\noindent
\begin{lemma}
For any partition $\lambda$ we have
\begin{equation}
\frac{p_{l}(\mathbf{cr}^{\lambda})}{l!}v_{\lambda} = \sum_{k=0}^{\infty} c_{2k-1} \FF_{l - ( 2k-1)}.v_{\lambda}.
\end{equation}
\begin{proof}
We can rephrase the result of Lascoux and Thibon (\cite{LT}, proposition 3.3) in our notation as
$$
 \left[ \frac{\tilde \EE_0(z)}{ \varsigma (z)} - E \right]
 v_{\lambda} = \sum_{l=1}^{\infty} \frac{p_l(\mathbf{cr}^{\lambda})}{l!} z^{l} v_{\lambda}.
$$
where by definition we have
$
\tilde \EE_0(z) := \sum_{k \in \Z + 1/2} e^{zk}E_{kk} = \sum_{r=1} \FF_r z^r,$ and $ \FF_1 = E$.
Therefore subtracting $E$ gets rid of the first term of the expansion
$
\frac{\tilde \EE_0(z)}{ \varsigma (z)} - E  = \sum_{l=1}z^l \left[ \sum_{k=0}c_{2k-1} \FF_{l-( 2k-1)}\right].
$
\end{proof}
\end{lemma}
\begin{example}
The first terms read
\begin{align*}
\frac{p_1(\mathbf{cr}^{\lambda})}{1!} v_{\lambda}&= \FF_2 v_{\lambda},
\qquad
\frac{p_2(\mathbf{cr}^{\lambda})}{2!}v_{\lambda} = \left[ \FF_3  - \frac{1}{24}\FF_1\right] v_{\lambda}, \qquad \frac{p_3(\mathbf{cr}^{\lambda})}{3!}v_{\lambda} = \left[ \FF_4  - \frac{1}{24}\FF_2\right] v_{\lambda},\\
\frac{p_4(\mathbf{cr}^{\lambda})}{4!}v_{\lambda} &= \left[ \FF_5  - \frac{1}{24}\FF_3 + \frac{7}{5760}\FF_1 \right] v_{\lambda}, \qquad \frac{p_5(\mathbf{cr}^{\lambda})}{3!}v_{\lambda} = \left[ \FF_6  - \frac{1}{24}\FF_4 + \frac{7}{5760}\FF_2\right] v_{\lambda}.\\
\end{align*}
\end{example}

\subsection{Boson-Fermion correspondence: $\mathcal{F}$ operators in terms of $\alpha$ operators}
\label{sec:BosonFermion}
The goal of this section is to express the operators we use in terms of sums of strings of $\alpha$ operators weighted by coefficients enriched with a geometric meaning (in fact Gromov-Witten invariants).
\par
First of all, it is well known that Boson-Fermion correspondence gives the following expression for $\FF_l $ in terms of $\alpha$ operators
\begin{equation}\label{eq:Fasalpha}
\FF_l = [z^l].\frac{1}{\varsigma(z)} \sum_{s > 0} \sum_{n,m \geq 0} \frac{1}{m!n!} \bigg(\sum_{\substack{k_1 + \cdots + k_m = s \\ k_i \geq 1}} \prod_{i = 1}^m \frac{\varsigma(k_iz)}{k_i} \alpha_{k_i}\bigg)\bigg(\sum_{\substack{\ell_1 + \cdots + \ell_{n} = s \\ \ell_i > 0}} \prod_{i = 1}^n \frac{\varsigma(\ell_i z)}{l_i}\alpha_{-{\ell_i}}\bigg).
\end{equation}
\noindent
Since $\FF_l^{sh} = \FF_l + c_l$, it is enough to add the $s=0$ term in the sum:
\begin{equation}
\FF_l^{sh} = [z^l].\frac{1}{\varsigma(z)} \sum_{s \geq 0} \sum_{n,m \geq 0} \frac{1}{m!n!} \bigg(\sum_{\substack{k_1 + \cdots + k_m = s \\ k_i \geq 1}} \prod_{i = 1}^m \frac{\varsigma(k_iz)}{k_i} \alpha_{k_i}\bigg)\bigg(\sum_{\substack{\ell_1 + \cdots + \ell_{n} = s \\ \ell_i > 0}} \prod_{i = 1}^n \frac{\varsigma(\ell_i z)}{l_i}\alpha_{-{\ell_i}}\bigg).
\end{equation}

\begin{definition}
Let $S\Z^{k + 1 - 2g }$ (resp.$\, S\Z^{k + 1 - 2g }_{sh}$) be the infinite subset of $\Z^{k + 1 - 2g }$ of integer vectors $\textbf{x}$ satisfying
$$
x_1 \leq \dots \leq x_l < 0 < x_{l+1} \leq \dots \leq x_{k+1 -2g}, \qquad \sum_{i=1}^{k + 1 - 2g} x_i = 0, \qquad l > 0 \;\;\,\, (\text{resp. } l \geq 0).
$$
Let $ \textbf{x}^-_i, \textbf{x}^+_i$ denote the partitions formed by the negative elements multiplied by a minus sign of $\textbf{x}$, and by the positive elements of $\textbf{x}$, respectively:
$$
\textbf{x}^-_i = -\textbf{x}_i = -x_i, \quad i=1, \dots, l, \qquad
\qquad
\textbf{x}^+_i = \textbf{x}_{l+i} = x_{l+i}, \quad i=1, \dots, k+1 - 2g - l.
$$
\end{definition}

Let us recall the following result, expressing $1$-point connected Gromov-Witten correlators with target $\mathbb{P}^1$ in terms of $\SS$ functions.
\begin{theorem}[\cite{OP}, Theorem 2]\label{thm:1ptGW} For any two partitions $\textbf{x}^+$ and $\textbf{x}^-$ of the same size
\begin{equation}
\cor{ \textbf{x}^+,   \tau_{2g - 2 + \ell(\textbf{x}^+) + \ell(\textbf{x}^-)} , \textbf{x}^-  }^{\mathbb{P}^1, \circ}_g
=
\frac{1}{|\Aut (\textbf{x}^+)||\Aut (\textbf{x}^-)|}[z^{2g}] \frac{\prod_{\textbf{x}^+_i} \SS(\textbf{x}_i z)\prod_{\textbf{x}^-_i} \SS(\textbf{x}_i z)}{\SS(z)}.
\end{equation}
\end{theorem}

\begin{lemma} \label{lem:Fasalpha}
For the operators $\mathcal{F}_{k}$ and $\mathcal{F}_{k}^{sh}$, we have the following identity
\begin{align}
\mathcal{F}_{k} &= \sum_{g=0}^{\infty} \sum_{\textbf{x} \, \in \, S\Z^{k + 1 - 2g}} 
\cor{ \textbf{x}^+,   \tau_{k-1}(\omega) , \textbf{x}^-  }^{\mathbb{P}^1, \circ}_g \prod_{0 > x_i \in \textbf{x}} \alpha_{x_i} \prod_{0 < x_j \in \textbf{x}} \alpha_{x_j},
\\
\mathcal{F}^{sh}_{k} &= \sum_{g=0}^{\infty} \sum_{\textbf{x} \, \in \, S\Z^{k + 1 - 2g}_{sh}} 
\cor{ \textbf{x}^+,   \tau_{k-1}(\omega) , \textbf{x}^-  }^{\mathbb{P}^1, \circ}_g \prod_{0 > x_i \in \textbf{x}} \alpha_{x_i} \prod_{0 < x_j \in \textbf{x}} \alpha_{x_j}.
\end{align}
\begin{proof}
Expand $\mathcal{F}_{k}$ by equation \eqref{eq:Fasalpha}. Observe that we obtain an infinite sum of words in $\alpha$ operators (or strings of $\alpha$ operators), weighted by certain coefficients. The longest possible word has length $k+1$, because $\varsigma(z)^{-1} = z^{-1} + O(z^1)$. Since $\varsigma(z)$ is an odd function in $z$, there are words of length $k$ with non-zero coefficients, and in general the length of the $\alpha$ strings decrease two by two. Let the index $g$ count half the defect of the $\alpha$ strings length. We obtain
\begin{equation*}
\mathcal{F}_{k} = \sum_{g=0}^{\infty} \sum_{\textbf{x} \in S\Z^{k + 1 - 2g}} 
\Bigg[ \frac{1}{|\Aut (\textbf{x}^+)||\Aut (\textbf{x}^-)|}\frac{[z^{k}]}{\prod_{\textbf{x}^+_i}\textbf{x}_i \prod_{\textbf{x}^-_i} \textbf{x}_i } \frac{\prod_{\textbf{x}^+_i} \varsigma(\textbf{x}_i z)\prod_{\textbf{x}^-_i} \varsigma(\textbf{x}_i z)}{\varsigma(z)}\Bigg]  \prod_{0 > x_i \in \textbf{x}} \!\!\!\! \alpha_{x_i} \!\!\!\! \prod_{0 < x_j \in \textbf{x}} \!\!\!\! \alpha_{x_j}
\end{equation*}
Note that $(m!n!)^{-1}$ in equation \eqref{eq:Fasalpha} gets substituted by $|\Aut (\textbf{x}^+)||\Aut (\textbf{x}^-)|$. This is because we are re-summing $m$-uplas and $n$-uplas in terms of partitions (hence elements are ordered) of length $m$ and $n$ respectively. 
 Substituting $\SS(z) = \varsigma(z)/z$ we get
\begin{equation*}
\mathcal{F}_{k} = \sum_{g=0}^{\infty} \sum_{\textbf{x} \in S\Z^{k + 1 - 2g}} 
\frac{1}{|\Aut (\textbf{x}^+)||\Aut (\textbf{x}^-)|}[z^{2g}] \frac{\prod_{\textbf{x}^+_i} \SS(\textbf{x}_i z)\prod_{\textbf{x}^-_i} \SS(\textbf{x}_i z)}{\SS(z)} \prod_{0 > x_i \in \textbf{x}} \alpha_{x_i} \prod_{0 < x_j \in \textbf{x}} \alpha_{x_j}.
\end{equation*}
Applying Theorem \ref{thm:1ptGW} proves the equation for $\mathcal{F}_k$. For $\mathcal{F}_k^{sh}$, simply note that the $s=0$ term corresponds to $[z^k]. \frac{1}{\varsigma(z)} = c_k = \langle \emptyset | \tau_{2k-2}(\omega) | \emptyset\rangle_k^{\mathbb{P}^1, \circ}$ and therefore corresponds to adding the empty $\alpha$ string, which is what the condition $l \geq 0$ takes care of.
\end{proof}
\end{lemma}
%

\begin{definition} \label{def:Gl}Let us define the operators
$
\mathcal{G}_{l+1} := (l -1)! \sum_{k=0}^{\infty} c_{2k - 1}\FF_{l  - (2k- 1)}.
$
\end{definition}
\begin{lemma}
For the operators $G_{l+1}$, we have the expression
\begin{equation}
\mathcal{G}_{l+1} = (l -1)! \!\!\!\!\!\!\!\!\!\!\!\! \sum_{\substack{g_1, g_2 = 0 
\\
\textbf{x} \in S\Z^{l + 2 - 2g_1 - 2g_2}}}^{\infty} 
\!\!\!\!\!\!\!\!\!\!\!\!\!
\cor{ \!\! \tau_{2g_2 - 2}(\omega) \!\! }_{g_2}^{\mathbb{P}^1, \circ} \cor{ \!\!\textbf{x}^+,   \tau_{2g_1 - 2 + \ell(\textbf{x}^+) + \ell(\textbf{x}^-)}(\omega) , \textbf{x}^-\!\!  }^{\mathbb{P}^1,\circ}_{g_1} \prod_{0 > x_i \in \textbf{x}} \!\!\!\! \alpha_{x_i} 
\!\!
\prod_{0 < x_j \in \textbf{x}} \!\!\!\! \alpha_{x_j}.
\end{equation}
\begin{proof}
The proof is a straightforward application of lemma \ref{lem:Fasalpha} and equation \eqref{eq:formscl}.
\end{proof}
\end{lemma}

\subsection{Putting the pieces together}

Putting together what we discussed in section \ref{sec:Newton} and in section \ref{sec:LascouxThibon} we obtain
\begin{align*}
h_m(\mathbf{cr}^{\rho})v_{\rho}
&= \sum_{\lambda \vdash m} \frac{1}{\ell(\lambda)!} \prod_i^{\ell(\lambda)} \left((\lambda_i - 1)!\sum_{k_i=0}^{\infty} c_{2k_i - 1}\FF_{\lambda_i  - (2k_i - 1)}\right).v_{\rho} \\
\sigma_m(\mathbf{cr}^{\rho})v_{\rho}
&= \sum_{\lambda \vdash m} \frac{(-1)^{m+ \ell(\lambda)}}{\ell(\lambda)!} \prod_i^{\ell(\lambda)} \left((\lambda_i - 1)!\sum_{k_i=0}^{\infty} c_{2k_i - 1} \FF_{\lambda_i -( 2k_i - 1)}\right).v_{\rho}.
\end{align*}
By the expressions of monotone and strictly monotone Hurwitz numbers in definition \eqref{eq:hfockh}
and \eqref{eq:hfocksigma} and by definition \ref{def:Gl} we have

\begin{align*}
h_{g; \mu, \nu}^{\leq, \bullet} 
& = \frac{1}{\prod_{i=1}^m\mu_i\prod_{j=1}^n\nu_j}
\sum_{\lambda \vdash m} \frac{1}{\ell(\lambda)!}
\cord{
\prod_{i=1}^{\ell(\mu)} \alpha_{\mu_j}
 \prod_i^{\ell(\lambda)} \mathcal{G}_{\lambda_i + 1}
\prod_{j=1}^{\ell(\nu)} \alpha_{-\nu_j}
}
\\
h_{g; \mu, \nu}^{<, \bullet} 
& = \frac{1}{\prod_{i=1}^m\mu_i\prod_{j=1}^n\nu_j}
\sum_{\lambda \vdash m} \frac{(-1)^{m+ \ell(\lambda)}}{\ell(\lambda)!}
\cord{
\prod_{i=1}^{\ell(\mu)} \alpha_{\mu_j}
 \prod_i^{\ell(\lambda)}  \mathcal{G}_{\lambda_i + 1}
\prod_{j=1}^{\ell(\nu)} \alpha_{-\nu_j}
}
\end{align*}

\section{Tropicalisation}
\label{sec:trop}
The goal of this section is to express connected monotone and strictly monotone Hurwitz numbers in terms of tropical covers weighted by Gromov-Witten invariants. This is achived by applying Wick's theorem to the expressions obtained from the bosonification. The main result of the paper is the following.

\begin{theorem}
\label{thm:trop}
Let $g$ be a non-negative integer, and $\mu,\nu$ partitions of the same size $d>0$. 

\begin{align}
h_{g; \mu, \nu}^{\leq, \circ}&=
\sum_{\pi \in \Gamma^{\circ}( \mathbb{P}^1_{\text{trop}}, g; \mu, \nu)}\frac{1}{|\mathrm{Aut}(\pi)|}\frac{1}{\ell(\lambda)!}\prod_{v \in V(\Gamma)} m_v \prod_{e \in E(\Gamma)} \omega_e
\\
h_{g; \mu, \nu}^{<, \circ}&=
\sum_{\pi \in \Gamma^{\circ}( \mathbb{P}^1_{\text{trop}}, g; \mu, \nu)}\frac{1}{|\mathrm{Aut}(\pi)|}\frac{1}{\ell(\lambda)!}\prod_{v \in V(\Gamma)} (-1)^{1 + val(v)} m_v \prod_{e \in E(\Gamma)} \omega_e
\end{align}
where $\Gamma^{\circ, \leq}( \mathbb{P}^1_{\text{trop}}, g; \mu, \nu)$ is the set of connected tropical covers 
$
\pi: \Gamma \longrightarrow \mathbb{P}^1_{trop} = \R
$
with $b=2g-2+\ell(\mu)+\ell(\nu)$ points $p_1,\dots,p_b$ fixed on the codomain $\mathbb{P}^1_{trop}$, such that
\begin{itemize}
\item[\textit{i).}] The unbounded left (resp. right) pointing ends of $\Gamma$ have weights given by the partition $\mu$ (resp. $\nu$).
\item[\textit{ii).}] There exists some $l\le b$, such that $\Gamma$ has $l$ many vertices. Let $V(\Gamma) = \{v_1, \dots, v_l\}$ be the set of its vertices. Then $\pi(v_i)=p_i$. Moreover, let $w_i  = \val(v_i)$ be the corresponding valences.
\item[\textit{iii).}] There are two integers associated to each vertex $v_i$ of $\Gamma$, $(g_1^i, g_2^i) \in \Z_{\geq 0}^2$ for $i=1, \dots, l$, such that we have $g(v_i)=g_1^i + g_2^i$ for the genus at $v_i$ and the following condition holds true

\begin{equation}
h^1(\Gamma) + \sum_{i=1}^l  g(v_i) = g.
\end{equation}


\item[\textit{iv).}] We define a partition $\lambda$ of length $l$ by $\lambda_i=\mathrm{val}(v_i)+2g(v_i)-2$ and impose the Riemann-Hurwitz condition 
\begin{equation}
\sum_{i=1}^l\lambda_i=2g-2+\ell(\mu)+\ell(\nu).
\end{equation}

\item[\textit{v).}]
For each vertex $v_i$, let $\textbf{x}^{+}$ (resp. $ \textbf{x}^-$) be the right-hand (resp. left-hand) side weights. The multiplicity $m_{v_i}$ of $v_i$ is defined to be
$$
m_{v_i} = (\lambda_i-1)!|\mathrm{Aut}(\textbf{x}^{+})||\mathrm{Aut}(\textbf{x}^{-})|
\left(
 \int_{\overline{\mathcal{M}}_{g_1^i, 1}(\mathbb{P}^1, \textbf{x}^{+}, \textbf{x}^-, |\textbf{x}^{+}|)} 
\psi^{2g_1^i - 2 + w_i} ev^{\star}_1(pt)
 \right)
 \!\!\!\!
 \left(
  \int_{\overline{\mathcal{M}}_{g_2^i, 1}} \!\!\!\!\lambda_{g_2^i} \psi^{2g_2^i - 2}
 \right)
$$
\end{itemize}
\end{theorem}

\begin{proof}
We will work out the details for the monotone case, as the strictly monotone case is completely parallel. By the previous section, we have
\begin{equation}
h_{g; \mu, \nu}^{\leq, \circ}  = 
\frac{1}{\prod_{i=1}^m\mu_i\prod_{j=1}^n\nu_j}\sum_{\lambda \vdash m} \frac{1}{\ell(\lambda)!}
\cor{
\prod_{i=1}^{\ell(\mu)} \alpha_{\mu_j}
 \prod_i^{\ell(\lambda)} \mathcal{G}_{\lambda_i + 1}
\prod_{j=1}^{\ell(\nu)} \alpha_{-\nu_j}
}^{\circ}.
\end{equation}

A generic summand of this expression is given by
\begin{align*}
&\frac{1}{\prod_{i=1}^m\mu_i\prod_{j=1}^n\nu_j}\frac{\prod (\lambda_i-1)!}{\ell(\lambda)!}\prod_{i=1}^{\ell(\lambda)}\cor{ \!\! \tau_{2g^i_2 - 2}(\omega) \!\! }_{g^i_2}^{\mathbb{P}^1, \circ}\times
\!\!\!\!
\\
&\cor{ \!\!\textbf{x}^{+,i},   \tau_{2g_1 - 2 + \ell(\textbf{x}^{+,i}) + \ell(\textbf{x}^{-,i})}(\omega) , \textbf{x}^-\!\!  }^{\mathbb{P}^1,\circ}_{g_1} \cor{\prod_{i=1}^{\ell(\mu)}\alpha_{\mu_i}\prod_{i=1}^{\ell(\lambda)}\alpha_{x^i_1}\cdots\alpha_{x^i_{\lambda_i+1-2g^i_1-2g^i_2}}\prod_{i=1}^{\ell(\nu)}\alpha_{\nu_i}}^{\circ},
\end{align*}
where $g^i_1,g^i_2\in\mathbb{Z}_{\ge0}$, $x^i\in S\mathbb{Z}^{\lambda_i+1-2g_1-2g_2}$ and $x^i_1\le\dots\le x^i_{l_i}<0<x^i_{l_i+1}\le\dots x^i_{\lambda_i+1-2g_1^i-2g_2^i}$ for some $l_i\in[\lambda_i+1-2g^i_1-2g^i_2]$ for all $i\in[\ell(\lambda)]$.\par
Let us now apply Wick's theorem \ref{prop:wick} to the last vacuum expectation obtaining consider the vacuum expectation
\begin{equation*}
\cor{\prod_{i=1}^{\ell(\mu)}\alpha_{\mu_i}\prod_{i=1}^{\ell(\lambda)}\alpha_{x^i_1}\cdots\alpha_{x^i_{\lambda_i+1-2g^i_1-2g^i_2}}\prod_{i=1}^{\ell(\nu)}\alpha_{\nu_i}}^{\circ}
\!\!\!
 = 
 \!\!\!\!\!
 \sum_{\pi \in \Gamma^{\circ}( \mathbb{P}^1_{\text{trop}}, 0; \mu, \nu)}\frac{1}{|\mathrm{Aut}(\pi)|} \prod_{i=1}^{\ell(\lambda)}|\mathrm{Aut}(x_i^+)||\mathrm{Aut}(x_i^-)|
 \!\!\!\! 
 \prod_{e \in E(\Gamma)} \!\!\!\! \omega_e,
\end{equation*}
These covers already satisfy conditions \textit{i).} and \textit{ii.)}. We note that since we are concerned with connected vacuum expectations, we only need to consider connected tropical covers.
\par
As observed in proposition \ref{prop:wick} the vertices of all tropical covers involved are of genus $0$. To relate these tropical covers to the desired ones, we associate new genera to the vertices. We start with a fixed generic expression with fixed data $g_1^i,g_2^i$ for $i=1,\dots,\ell(\lambda)$ for a fixed $\lambda$, such that $|\lambda|=2g-2+\ell(\mu)+\ell(\nu)$.\par
We associate a new tropical cover to each cover involved in equation \eqref{equ:trop} by setting $g(v_i)=g_1^i+g_2^i$. This obviously a bijection and it preserves automorphisms. We now check that $h^1(\Gamma)+\sum_{i=1}^lg(v_i)=g$. Recall that the Euler characteristic for graphs reads $|V|-|E|=1-h^1(\Gamma)$, where $|V|$ is the number of vertices and $|E|$ is the number of edges. We observe $|V|=\ell(\mu)+\ell(\nu)+\ell(\lambda)$. Moreover, by the handshake lemma, we obtain
\begin{align}
|E|=\frac{1}{2}\sum_{v\in V(\Gamma)}\mathrm{val}(v)&=\frac{1}{2}\left(\sum_{i=1}^{\ell(\lambda)}(\lambda_i+2-2g(v_i))+\ell(\mu)+\ell(\nu)\right)\\
&=\frac{1}{2}\left(|\lambda|+2\ell(\lambda)-2\sum g(v_i)+\ell(\mu)+\ell(\nu)\right)
\end{align}
and substituting $|\lambda|=2g-2+\ell(\mu)+\ell(\nu)$, we obtain
\begin{equation}
|E|=\frac{1}{2}\left(2g-2+2\ell(\mu)+2\ell(\nu)+2\ell(\lambda)-2\sum_{i=1}^lg(v_i)\right)=g-1+\ell(\mu)+\ell(\nu)+\ell(\nu)-\sum_{i=1}^lg(v_i).
\end{equation}
Thus by imposing the Euler characteristic constraint, we obtain
\begin{equation}
1-h^1(\Gamma)=(\ell(\mu)+\ell(\nu)+\ell(\lambda))-(g-1+\ell(\mu)+\ell(\nu)+\ell(\nu)-\sum_{i=1}^lg(v_i))=1-g+\sum_{i=1}^lg(v_i).
\end{equation}
Thus, we obtain $g=h^1(\gamma)+\sum_{i=1}^l(g_1^i+g_2^i)$ as required, and the new associated tropical covers satisfy condition \textit{iii.)}. Condition \textit{iv).} is fulfilled by construction.\par
Now, we incorporate the prefactor of the above generic summand into the tropical cover as global and local multiplicities. Recall that the global prefactor is
\begin{align*}
&\frac{1}{\prod_{i=1}^m\mu_i\prod_{j=1}^n\nu_j}\frac{\prod (\lambda_i-1)!}{\ell(\lambda)!}\prod_{i=1}^{\ell(\lambda)}\cor{ \!\! \tau_{2g^i_2 - 2}(\omega) \!\! }_{g^i_2}^{\mathbb{P}^1, \circ} \cor{ \!\!\textbf{x}^{+,i},   \tau_{2g_1 - 2 + \ell(\textbf{x}^{+,i}) + \ell(\textbf{x}^{-,i})}(\omega) , \textbf{x}^-\!\!  }^{\mathbb{P}^1,\circ}_{g_1}
\end{align*}
whereas the tropical covers to be weighted by
\begin{equation}
\label{equ:trop}
\sum_{\pi \in \Gamma( \mathbb{P}^1_{\text{trop}}; \mu, \nu)}\frac{1}{|\mathrm{Aut}(\pi)|} \prod_{i=1}^n|\mathrm{Aut}(x_i^+)||\mathrm{Aut}(x_i^-)|\prod_{e \in E(\Gamma)} \omega_e.
\end{equation}
We define the multiplicity of the $i-$th vertex by
\begin{equation}
m_{v_i}=(\lambda_i-1)!|\mathrm{Aut}(x_i^+)||\mathrm{Aut}(x_i^-)\cor{ \!\! \tau_{2g^i_2 - 2}(\omega) \!\! }_{g^i_2}^{\mathbb{P}^1, \circ} \cor{ \!\!\textbf{x}^{+,i},   \tau_{2g_1 - 2 + \ell(\textbf{x}^{+,i}) + \ell(\textbf{x}^{-,i})}(\omega) , \textbf{x}^-\!\!  }^{\mathbb{P}^1,\circ}_{g_1}
\end{equation}
and we weight each cover by
\begin{equation}
\frac{1}{\prod\mu_i\prod\nu_j}\frac{1}{\ell(\lambda)!}\frac{1}{|\mathrm{Aut}(\pi)|}\prod_{e\in E(\Gamma)}\omega_e\prod m_{v_i}
\end{equation}
to obtain the same contribution. Finally, we note that the weights of the unbounded edges of each cover contribute a factor of $\prod\mu_i\prod\nu_j$ in the last product, which yields the weight of each cover to be
\begin{equation}
\frac{1}{\ell(\lambda)!}\frac{1}{|\mathrm{Aut}(\pi)|}\prod\omega_e\prod m_{v_i},
\end{equation}
where the product is taken over all inner edges of $\Gamma$. This concludes the proof of the theorem.
\end{proof}

\subsection{Comparison with the tropical curves obtained in \cite{DK,hahn2017monodromy}}
In this section, we compare the tropical curves obtained in theorem \ref{thm:trop} to the ones derived in \cite{DK,hahn2017monodromy} by means of a particular example. Both these papers use the notion of \textit{monotone monodromy graphs associated to the data }$(g,\mu,\nu)$ (for a formal definition of such graphs we refer to the original papers). For us is relevant that such graphs carry extra combinatorial data \textemdash \, including each edge being bi-labelled and being coloured in either \textit{solid, dashed, or bold}.

\begin{theorem}[\cite{DK,hahn2017monodromy}]
For $g\ge0$ and $\mu,\nu$ partitions of the same size, we have
\begin{equation}
h_{g;\mu,\nu}^{\le,\circ}=\frac{1}{\prod\mu_i}\sum_{G}\frac{1}{|\mathrm{Aut}(G)|}\prod\omega_e,
\end{equation}
where we sum over all monodromy graphs associated to $(g,\mu,\nu)$. In this notation $\omega_e$ denotes the weight of an edge $e$ of $G$ and the product is taken over a specific subset $H$ of the edges of $G$.
\end{theorem}

\begin{example}[For the data $(g,\mu,\nu)=(0,(2,1),(3))$]
The data $(g,\mu,\nu)=(0,(2,1),(3))$ produces only two monodromy graphs, illustrated at the top of figure \ref{fig:mono}. They do not have any automorphism, and their subset $H$ consists in both cases of the dashed edge. The weights of the edges are illustrated as numbers in round brackets. Thus we obtain
\begin{equation}
h_{0;(2,1),(3)}^{\le,\circ}=\frac{1}{1\cdot 2}\cdot1+\frac{1}{1\cdot 2}\cdot1=1.
\end{equation}
\noindent
In the language of theorem \ref{thm:trop}, we obtain the same result by considering a single tropical cover. Its graph is illustrated at the bottom of figure \ref{fig:mono} \textemdash \, it does not have any automorphism and corresponds to $\lambda=(1)$, which clearly gives $|\ell(\lambda)!|=1$. The local vertex multiplicity can be computed using the relation $c_{2l-1}= \left \langle \tau_{2l-2}(\omega) \right \rangle_{l,1}^{\mathbb{P}^1}$ and theorem \ref{thm:1ptGW}. This yields a vertex multiplicity of $1$ and thus

\begin{equation}
h_{0;(2,1),(3)}^{\le,\circ}=\frac{1}{1}\cdot \frac{1}{1}\cdot 1=1.
\end{equation}

\begin{figure}
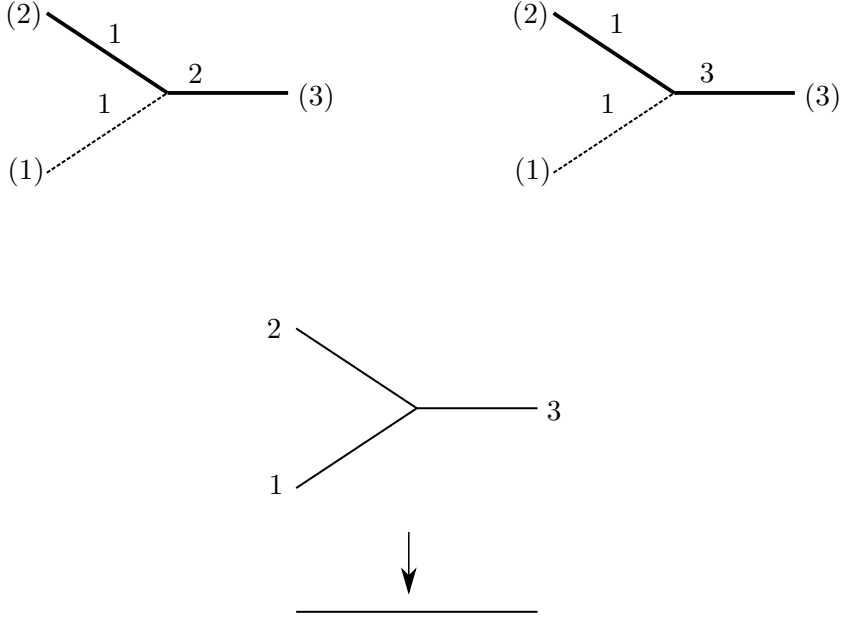
\caption{Comparing the tropical curves appearing in \cite{DK,hahn2017monodromy} (top) to the tropical curves appearing in theorem \ref{thm:trop} (bottom) for $(g,\mu,\nu)=(0,(2,1),(3))$}
\label{fig:mono}
\end{figure}
\end{example}

We now compute a slightly more involved example.

\begin{example}
We consider the case $(g,\mu,\nu)=(0,(3,1),(3,1))$. Using GAP \cite{GAP4}, we compute the monotone Hurwitz number in terms of factorisations in the symmetric group and obtain
\begin{equation}
h_{0;(2,1),(3)}^{\le,\circ}=4.
\end{equation}
We now show that our interpretation in theorem \ref{thm:trop} gives the same number. The three connected graphs for the data $(g,\mu,\nu)=(0,(3,1),(3,1))$ are illustrated in figure \ref{fig:ex2}.
\begin{enumerate}
\item For the graph on the left in figure \ref{fig:ex2}, we obtain $|\mathrm{Aut}(\pi)|=1$. Further, we see that $\lambda=(1,1)$ and thus $\ell(\lambda)!=2$. For both vertices, we obtain $m_v=1$. The only inner edge is of weight $\omega(e)=4$, which yields a contribution of $\frac{1}{1}\cdot\frac{1}{2}\cdot 1\cdot 4=2$.
\item For the graph in the middle of figure \ref{fig:ex2}, we obtain $|\mathrm{Aut}(\pi)=1$. Again, we obtain $\lambda=(1,1)$ and $\ell(\lambda)=2$. For both vertices, we obtain $m_v=1$ and the only inner edge is of weight $\omega(e)=2$. This yields a contribution of $\frac{1}{1}\cdot\frac{1}{2}\cdot1\cdot 2=1$.
\item For the graph on the right in figure \ref{fig:ex2}, we obtain $|\mathrm{Aut}(\pi)=1,\lambda=(2)$ and $\ell(\lambda)!=1$. The only vertex gives $m_v=1$ and there are no inner edges. Thus, we obtain a contribution of $\frac{1}{1}\cdot\frac{1}{1}\cdot1\cdot1=1$.
\end{enumerate}
In total, we obtain $h_{0;(3,1),(3,1)}^{\le,\circ}=2+1+1=4$, which is what we expected.

\begin{figure}
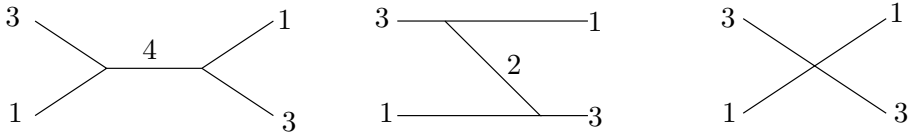
\caption{The tropical curves for $(g,\mu,\nu)=(0,(3,1),(3,1))$ according to theorem \ref{thm:trop}.}
\label{fig:ex2}
\end{figure}

\end{example}

\section{Gromov-Witten theory and tropical curves}
\label{sec:grom}
The goal of this section is to analyse and compare the original \cite{OP} and the tropical \cite{CJMR} version of Gromov-Witten/Hurwitz correpondence via the semi-infinite wedge formalism. The results are summarised at the end of the section.
%


\begin{definition}
Let $M_l(u)$ be the operator
\begin{equation*}
\frac{[z^{l+1}].}{u^{1/2}\varsigma(z u^{1/2})} \sum_{s \geq 0} \sum_{n,m \geq 0} \frac{1}{m!n!} \bigg(\sum_{\substack{\sum_{i=1}^m k_i = s \\ k_i \geq 1}} \prod_{i = 1}^m \frac{u^{-1/2}\varsigma(k_iz u^{3/2})}{k_i} \alpha_{k_i}\bigg)
\!\!
\bigg(\sum_{\substack{\sum_{i=1}^n = s \\ \ell_i \geq 1}} \prod_{i = 1}^n \frac{u^{-1/2}\varsigma(\ell_i z u^{1/2})}{l_i}\alpha_{-{\ell_i}}\bigg).
\end{equation*}
\end{definition}
\noindent
For us it will be relevant that
\begin{equation}\label{eq:MisFsh}
M_l(1) \!\! = \mathcal{F}_{l+1}^{sh}.
\end{equation}
After expanding it in the same way as in Lemma \ref{lem:Fasalpha}, $M_l$ takes the form
\begin{equation}\label{eq:expansionM}
M_l(u) =  \sum_{g=0}^{\infty} \sum_{\textbf{x} \, \in \, S\Z^{l + 2 - 2g}_{sh}} 
\cor{ \textbf{x}^+,   \tau_{l}(\omega) , \textbf{x}^-  }^{\mathbb{P}^1, \circ}_g 
\!\!\!\! 
u^{|\textbf{x}^+| + g -1} 
\prod_{0 > x_i \in \textbf{x}} \!\!\!\! \alpha_{x_i} 
\prod_{0 < x_j \in \textbf{x}} \!\!\!\! \alpha_{x_j}.
\end{equation}
and therefore coincide with the operator $M_l$ defined in \cite{CJMR}. We are now ready to list and compare the two results.

\begin{theorem}[\cite{CJMR}, Theorem 5.3.4]\label{thm:CJMR}
\begin{equation}
|\Aut(\mu)||\Aut(\nu)|\corcgw{\mu, \prod_{i=1}^n \tau_{k_i}(\omega), \nu} = \frac{[u^{g + \ell(\mu) - 1}]}{\prod_i \mu_i \prod_j \nu_j}
\corc{
\prod_{i=1}^{\ell(\mu)} \alpha_{\mu_j}
 \prod_i^{n} M_{k_i}(u)
\prod_{j=1}^{\ell(\nu)} \alpha_{-\nu_j}
}.
\end{equation}
\end{theorem}

\begin{theorem}[\cite{OP}, GW/H correspondence]
\label{thm:CJMRu1}
\begin{equation}
|\Aut(\mu)||\Aut(\nu)|\corcgw{\mu, \prod_{i=1}^n \tau_{k_i}(\omega), \nu} = \frac{1}{\prod_i \mu_i \prod_j \nu_j}
\corc{
\prod_{i=1}^{\ell(\mu)} \alpha_{\mu_j}
 \prod_i^{n} M_{k_i}(1)
\prod_{j=1}^{\ell(\nu)} \alpha_{-\nu_j}
}
\end{equation}
\end{theorem}
\begin{remark}
We rearranged GW/H correspondence in a form that is more convenient to our purposes. To derive the formulation above from the original paper \cite{OP}, simply combine their Equation 3.2 with their Proposition 3.1 obtaining
\begin{equation}\label{eq:GWH}
|\Aut(\mu)||\Aut(\nu)|\cordgw{\mu, \prod_{i=1}^n \tau_{k_i}(\omega), \nu} = 
\frac{1}{\prod_i \mu_i \prod_j \nu_j}
\cord{
\prod_{i=1}^{\ell(\mu)} \alpha_{\mu_j}
 \prod_i^{n} \mathcal{F}_{k_i + 1}^{sh}
\prod_{j=1}^{\ell(\nu)} \alpha_{-\nu_j}
},
\end{equation}
then conclude by equation \eqref{eq:MisFsh} and by taking the connected correlators on both sides.
\end{remark}
\noindent
Clearly, in order for both results to hold true, some non-trivial property of the correlator 
$$
\corc{
\prod_{i=1}^{\ell(\mu)} \alpha_{\mu_j}
 \prod_i^{n} M_{k_i}(u)
\prod_{j=1}^{\ell(\nu)} \alpha_{-\nu_j}
}
$$
considered as formal power series $u$ should be involved \textemdash \, in fact, either the degree is completely concentrated in $g + \ell(\mu) - 1$, or the sum of all the coefficients, excluded the coefficient of $u^{g + \ell(\mu) - 1}$, should vanish altogether. We are going to prove that the first is the correct one. 

\begin{proposition}[Concentration of the degree] Let $K_0$ be $\frac{ \sum k_i + \ell(\mu) - \ell(\nu) }{2}$.
\begin{equation}
[u^K]\corc{
\prod_{i=1}^{\ell(\mu)} \alpha_{\mu_j}
 \prod_i^{n} M_{k_i}(u)
\prod_{j=1}^{\ell(\nu)} \alpha_{-\nu_j}
} = 0 \qquad \qquad \text{ for } K \neq K_0.
\end{equation}
\end{proposition}
\begin{remark}
To avoid confusion, let us remark that the statement of the proposition expresses $K_0$ as $( \sum k_i + \ell(\mu) - \ell(\nu)) / 2$ instead of $g + \ell(\mu) - 1$ because the variable $g$ loses its meaning without mentioning the GW correlator \textemdash\, however, the two quantities are equal since the genus is determined by the Riemann-Hurwitz condition $\sum k_i = 2g - 2 + \ell(\mu) + \ell(\nu)$.
\end{remark}
The proposition immediately implies the following corollary.
\begin{corollary}\label{cor:equiv5152}
Theorem \ref{thm:CJMR} holds true if and only if 
Theorem \ref{thm:CJMRu1} holds true. 
\end{corollary}
\begin{proof}
To be precise, what we are going to show is that \emph{each vacuum expectation} produced by the product of the formal expansions of the $M_{k_i}(u)$ in degree different from $K_0$ vanishes on its own (i.e. there are no non-trivial cancellations even between summands of the same degree, if that degree differs from $K_0$). 
Let us analyse a generic summand of the vacuum expectation
$
\langle
\prod_{i=1}^{\ell(\mu)} \alpha_{\mu_j}
 \prod_i^{n} M_{k_i}(u)
\prod_{j=1}^{\ell(\nu)} \alpha_{-\nu_j}
\rangle
$
expanded via equation \eqref{eq:expansionM}. A generic such summand is of the form
\begin{equation*}
\prod_{i=1}^n \!\! \cor{ \!\! \textbf{x}^{+, (i)},   \tau_{k_i}(\omega) , \textbf{x}^{-, (i)}  \!\!}^{\mathbb{P}^1, \circ}_{g_i'}
\!\!
\cor{\!\!
\prod_{t=1}^{\ell(\mu)} \!\!\alpha_{\mu_t}
\prod_{i=1}^n 
\alpha_{- \textbf{x}^{-, (i)}_1} \cdots \alpha_{- \textbf{x}^{-, (i)}_{p_i}}
\cdot
\alpha_{ \textbf{x}^{+, (i)}_1} \cdots \alpha_{ \textbf{x}^{+, (i)}_{q_i}}\!\!
\prod_{j=1}^{\ell(\nu)} \alpha_{-\nu_j}\!\!
}  u^{\sum_{i=1}^n (p_i + g_i' - 1)} .
\end{equation*}
Our goal is to show that every non-vanishing summand satisfies the equation
\begin{equation}
\sum_{i=1}^n (p_i + g_i' - 1) = g + \ell(\mu) - 1.
\end{equation}

We provide two different proofs of the concentration of the degree: the first involves the analysis of each vacuum expectation via the commutation relations of the $\alpha$ operators, the second one analyses the tropical curves associated to each vacuum expectation and compute their Euler characteristic.

\begin{proof}[Proof via vacuum expectations analysis]
The vanishing of each such summand comes from the following lemma, which immediately follows from \eqref{eq:pairing} and the commutation relations \eqref{eq:commalphas}. It also follows as immediate corollary from condition \textit{iii).} of Wick's theorem \ref{prop:wick}.
\begin{lemma}\label{lem:vanishingcondition} For any collection of finite sets of non-zero integers $\textbf{x}_1, \dots, \textbf{x}_n$, the vacuum expectation
\begin{equation}
\cor{
\prod_{t=1}^{\ell(\mu)} \alpha_{\mu_t} 
\prod_{0 > x_{i,1} \in \textbf{x}_1} \!\!\!\!\!\!\alpha_{x_{i,1}}
\prod_{0 < x_{i,1} \in \textbf{x}_1} \!\!\!\!\!\!\alpha_{x_{i,1}}
\cdots \cdots
\prod_{0 > x_{i,n} \in \textbf{x}_n} \!\!\!\!\!\!\alpha_{x_{i,n}}
\prod_{0 < x_{i,n} \in \textbf{x}_n} \!\!\!\!\!\!\alpha_{x_{i,n}}
\prod_{j=1}^{\ell(\nu)} \alpha_{-\nu_j}
} = 0
\end{equation}
whenever the following condition is not satisfied:
\begin{equation}
\left\{ \mu_i \right\}_{i = 1, \dots, \ell(\mu)} \cup \textbf{x}_1^+ \cup \dots \cup \textbf{x}_n^+
=
\left\{ \nu_i \right\}_{i = 1, \dots, \ell(\nu)} \cup \textbf{x}_1^- \cup \dots \cup \textbf{x}_n^- .
\end{equation}
In other words, the set of positive indices  appearing in the expression must be equal to the set of the absolute values of the negative indices.
\end{lemma}

\noindent
From lemma \ref{lem:vanishingcondition}, we know that, for every non-vanishing summand, the set of positive indices must be equal to the set of the absolute values of the negative indices. In particular, the cardinalities of the two sets must be equal. This gives us the additional requirement
\begin{equation}\label{eq:nonvanishingcondition}
\sum_{i=1}^{n} p_i + \ell(\nu) =  \sum_{i=1}^{n} q_i + \ell(\mu) \qquad \text{ or, equivalently, } \qquad \sum_{i=1}^{n} p_i - \ell(\mu) =  \sum_{i=1}^{n} q_i - \ell(\nu).
\end{equation}
Now, from the definition of $M_{k_i}(u)$ operators we get $p_i + q_i = k_i + 2 - 2g_i'$, and therefore
\begin{equation}
\sum_{i=1}^n p_i + \sum_{i=1}^n q_i = \sum_{i=1}^n k_i + 2n - 2\sum_{i=1}^n g_i'.
\end{equation}
Imposing the Riemann-Hurwitz constraint $\sum_{i=1}^n k_i = 2g - 2 + \ell(\mu) + \ell(\nu)$, we get
\begin{equation}
 2\left(g - 1 -\sum_{i=1}^n g_i'\right) = \sum_{i=1}^n p_i - \ell(\mu) +  \sum_{i=1}^n q_i - \ell(\nu) - 2n .
\end{equation}
Applying the additional requirement \eqref{eq:nonvanishingcondition} and dividing by two gives the desired equation
\begin{equation}
 g - 1 - \sum_{i=1}^n g_i' =
  \sum_{i=1}^n p_i - \ell(\mu) - n
  \qquad \text{ or, equivalently, } \qquad
  \sum_{i=1}^n (p_i + g_i' - 1) = g + \ell(\mu) - 1.
\end{equation}
\end{proof}

\begin{proof}[Proof via tropical curves' Euler characteristic analysis]
Let $\pi: \Gamma \longrightarrow \mathbb{P}^1_{trop} = \R$ be the tropical curve associated to the generic summand as above. Let $V(\Gamma), E(\Gamma), h^1(\Gamma)$ the set of vertices, the set of edges and the first Betti number of the tropical cover $\Gamma$. Because we are considering connected covers, the Euler characteristic constraint gives
\begin{equation}
|V(\Gamma)| - |E(\Gamma)| = 1  - h^1(\Gamma).
\end{equation}
Let us compute each ingredient separately in terms of the indices of the generic correlator.
\begin{enumerate}
\item[\textit{i).}] $|V(\Gamma)| = \ell(\mu) + \ell(\nu) + n$. 
\item[\textit{ii).}] $|E(\Gamma)| = \sum_{i=1}^n p_i + \ell(\nu)$. Indeed the sum of the $p_i$ counts the incoming edges, and every edge is either an incoming edge for some vertex, or it is an edge of infinite length on the right and therefore must correspond to some part of $\nu$.
\item[\textit{iii).}] $h^1(\Gamma) = g - \sum_{i=1}^n g_i$.
\end{enumerate}
Substuting these quantities in the Euler characteristic contraint we obtain
\begin{equation}
\ell(\mu) + \ell(\nu) + n - \left(\sum_{i=1}^n p_i  + \ell(\nu) \right)
=
1 - \left( g - \sum_{i=1}^n g_i \right),
\end{equation}
which is equivalent to 
\begin{equation}
g + \ell(\mu) - 1  = \sum_{i=1}^n \big(p_i  + g_i - 1\big).
\end{equation}
\end{proof}
\noindent
This concludes the proof of the proposition.
\end{proof}
We can summarise the results of this section in the following two points.
\begin{enumerate}
\item[i).] It provides another proof of Theorem \ref{thm:CJMR} (by combining Theorem \ref{thm:CJMRu1} and Corollary \ref{cor:equiv5152}). 
\item[ii).] It allows a clear comparison between Theorem \ref{thm:CJMR} and GW/H correspondence \ref{thm:CJMRu1} \textemdash \, on the one hand, it shows that Theorem \ref{thm:CJMR} is a \emph{refined} version of the GW/H correspondence, since it discards many summands that contribute trivially and otherwise would be counted. On the other hand, it shows that the formal $u$-variable in Theorem \ref{thm:CJMR} is redundant, which is a non trivial fact, and that getting rid of the $u$-variable recovers nothing but GW/H correspondence.
\end{enumerate}

\bibliography{literature.bib}
\bibliographystyle{alpha}

\end{document}